\documentclass[11pt, reqno]{amsart}

\usepackage[T1]{fontenc}
\usepackage{amsmath}
\usepackage{amssymb}
\usepackage{amsthm}
\usepackage[left=3.5cm, right=3.5cm, paperheight=11.9in]{geometry}
\usepackage{hyperref}
\usepackage{fancyhdr}
\usepackage{enumitem}
\usepackage{comment}
\usepackage{nicefrac}
\usepackage{bm}
\usepackage{mathrsfs}
\usepackage{graphicx}
\usepackage[utf8]{inputenc}
\usepackage{cancel}
\usepackage{mathtools}

\author[P.~Leonetti]{Paolo Leonetti}
\address{Department of Economics, Universit\`a degli Studi dell'Insubria, via Monte Generoso 71, 21100 Varese, Italy}
\email{leonetti.paolo@gmail.com}
\urladdr{\url{https://sites.google.com/site/leonettipaolo/}}

\newtheorem{thm}{Theorem}[section]
\newtheorem{cor}[thm]{Corollary}

\theoremstyle{definition} 
\newtheorem{defi}[thm]{Definition}

\numberwithin{equation}{section}

\newcounter{smallromans}

{\end{list}}

\AtBeginDocument{%
   \def\MR#1{}
}

\hypersetup{
    pdftitle={Almost all sets and their perturbations are not sumsets},
    pdfauthor={Paolo Leonetti},
    pdfmenubar=false,
    pdffitwindow=true,
    pdfstartview=FitH,
    colorlinks=true,
    linkcolor=blue,
    citecolor=green,
    urlcolor=cyan
}

\providecommand{\MR}[1]{}

\providecommand{\MR}{\relax\ifhmode\unskip\space\fi MR }

\providecommand{\href}[2]{#2}

\begin{document}
\title{Almost all sets of nonnegative integers and their small perturbations are not sumsets}


\keywords{Sumset; irreducible set; totally irreducible set; meager set.}
\subjclass[2020]{Primary: 11B13, 54E52; Secondary: 11B05, 11B30.}

\begin{abstract} 
Fix $\alpha \in (0,1/3)$. We show that, from a topological point of view, almost all sets $A\subseteq \mathbb{N}$ have the property that, if $A^\prime=A$ for all but $o(n^{\alpha})$ elements, then $A^\prime$ is not a nontrivial sumset $B+C$. In particular, almost all $A$ are totally irreducible. In addition, we prove that the measure analogue holds with $\alpha=1$. 
\end{abstract}
\maketitle
\thispagestyle{empty}

\section{Introduction}\label{sec:intro}

A subset $A$ of the nonnegative integers $\mathbb{N}$ is said to be \emph{irreducible} if there do not exist $B,C\subseteq \mathbb{N}$ such that $|B|,|C|\ge 2$ and 
$$
A=B+C,
$$
where $B+C:=\{b+c: b \in B, c \in C\}$. 
In addition, we say that $A\subseteq \mathbb{N}$ is \emph{totally irreducible} (or \emph{totally primitive}) if there are no sets $B,C\subseteq \mathbb{N}$ such that $|B|,|C|\ge 2$ and
$$
A=_\star B+C
$$
meaning that the symmetric difference $A\bigtriangleup (B+C)$ is finite, cf. \cite{Bien2022, MR58655}; equivalently, $A$ and $B+C$ belong to the same equivalence class in $\mathcal{P}(\mathbb{N})/\mathrm{Fin}$, where $\mathrm{Fin}$ stands for the family of finite subsets of $\mathbb{N}$. 
An old conjecture of Ostmann \cite[p. 13]{MR0268112}, which is still open, states that the set of primes is totally irreducible, cf. \cite{MR2271607, MR3378834} for partial results. 
See also \cite{Bien2022, Ruzsa2022} for sumsets of symmetric sets of integers $S\subseteq \mathbb{Z}$ and references therein. 

Let $\Sigma$ and $\Sigma_\star$ be the family of irreducible and totally irreducible sets, respectively, so that $\Sigma_\star\subseteq \Sigma$. 
Also, identify the family of infinite sets $S\subseteq \mathbb{N}$ with the set of reals in $(0,1]$ through their unique nonterminating dyadic expansions. Also, denote by $\lambda: \mathscr{M} \to \mathbb{R}$ the Lebesgue measure, where $\mathscr{M}$ stands for the completion of the Borel $\sigma$-algebra on $(0,1]$. Accordingly, a theorem of Wirsing \cite{MR58655} states that almost all infinite sets are totally irreducible, in the measure theoretic sense:
\begin{thm}\label{thm:wirsing}
$\lambda(\Sigma_\star)=1$. In particular, also $\lambda(\Sigma)=1$.
\end{thm}

Results on the same spirit have been studied by S{\'{a}}rk{\"{o}}zy \cite{MR568278, MR568279}  for the Hausdorff dimension. Another related result has been recently obtained by Bienvenu and Geroldinger in \cite[Theorem 6.2]{Bien2022b}. Here, once we identify $\mathcal{P}(\mathbb{N})$ with the Cantor space $\{0,1\}^{\mathbb{N}}$, we show that the analogue of Theorem \ref{thm:wirsing} holds in the category sense, so that almost all sets are totally primitive also topologically (recall that a set is said to be comeager if its complement is of the first Baire category, namely, its complement is contained in a countable union of closed sets with empty interior): 
\begin{thm}\label{thm:meageranalogue}
$\Sigma_\star$ is a comeager subset of $\mathcal{P}(\mathbb{N})$. In particular, also $\Sigma$ is comeager.
\end{thm}

Our Theorem \ref{thm:meageranalogue} will be obtained as a consequence of a stronger result, which need some additional notation. 
\begin{defi}
Given a family $\mathcal{I}\subseteq \mathcal{P}(\mathbb{N})$ closed under finite unions and subsets, a set $A\subseteq \mathbb{N}$ is said $\mathcal{I}$\emph{-irreducible} if there are no $B,C\subseteq \mathbb{N}$ such that $|B|,|C|\ge 2$ and 
$$
A=_{\mathcal{I}}B+C,
$$
meaning that the symmetric difference $A\bigtriangleup (B+C) \in \mathcal{I}$. The family of $\mathcal{I}$-irreducible sets is denoted by $\Sigma(\mathcal{I})$.
\end{defi}

Note that $\emptyset$-irreducible and $\mathrm{Fin}$-irreducible are simply the classical irreducible and totally irreducible sets, respectively; 
similarly, also $\Sigma(\emptyset)=\Sigma$ and $\Sigma(\mathrm{Fin})=\Sigma_\star$. 
It is clear that $\Sigma(\mathcal{I}) \subseteq \Sigma(\mathcal{J})$ whenever $\mathcal{J}\subseteq \mathcal{I}$. Given $\alpha \in (0,1]$, define the family 
$$
\mathcal{Z}_\alpha:=\left\{S\subseteq \mathbb{N}: \lim_{n\to \infty} \frac{|S \cap [0,n]|}{n^\alpha}=0\right\}.
$$
Hereafter, we will use the shorter notations $S(n):=S\cap [0,n]$ and $\mathcal{Z}:=\mathcal{Z}_1$. 
Of course, $\mathcal{Z}$ is simply the zero set of the upper asymptotic density on $\mathbb{N}$, cf. \cite{MR4054777}. 
Structural properties of the families $\mathcal{Z}_\alpha$ are studied, e.g., in \cite{MR3391516, MR3771234, MR3950052}. 
With this notation, Erd{\H{o}}s conjectured that the set $Q:=\{n^2: n \in \mathbb{N}\}$ of nonnegative squares belongs to $\Sigma(\mathcal{Z}_{1/2})$. 
Then, S{\'{a}}rk{\"{o}}zy and Szemer\'{e}di proved in \cite{MR201410} a slightly weaker statement, namely, $Q \in \Sigma(\mathcal{Z}_\alpha)$ for all $\alpha \in (0,\frac{1}{2})$.

With the same spirit of Erd{\H{o}}s' conjecture, S{\'{a}}rk{\"{o}}zy and Szemer\'{e}di's result, and the claimed analogue stated in Theorem \ref{thm:meageranalogue}, we show that, from a topological point of view, almost all sets belong to $\mathcal{Z}_{\alpha}$, provided that $\alpha$ is sufficiently small: 
\begin{thm}\label{thm:mainmeager}
$\Sigma(\mathcal{Z}_{\alpha})$ is a comeager subset of $\mathcal{P}(\mathbb{N})$ for each $\alpha \in (0,\frac{1}{3})$. 
\end{thm}
Observe that Theorem \ref{thm:meageranalogue} is now immediate since $\Sigma(\mathcal{Z}_{1/4})\subseteq \Sigma_\star$. 
Results in the same spirit, but completely different contexts, appeared, e.g., in \cite{
AveniLeo22, 
MR4453356, 
Leo19, 
LeoTAUB, 
LeoAmir22}. 

In addition, we prove that the measure analogue  Theorem \ref{thm:mainmeager} holds, hence providing a generalization of Wirsing's Theorem \ref{thm:wirsing}:
\begin{thm}\label{thm:mainmeasure}
$\lambda(\Sigma(\mathcal{Z}))=1$.
\end{thm}

It is remarkable that S{\'{a}}rk{\"{o}}zy and Szemer\'{e}di proved in \cite{MR201410} a general criterion for a set $S\subseteq \mathbb{N}$ and all its small perturbations to be totally irreducible.  
However, the hypotheses of such result do not seem to apply in our case for a proof of Theorem \ref{thm:mainmeasure}.

Before we proceed to the proofs of Theorem \ref{thm:mainmeager} and Theorem \ref{thm:mainmeasure}, some remarks are in order. 
First, suppose that the family $\mathcal{I}$ contains the finite sets $\mathrm{Fin}$.  Since $\{0,1\}+\{0,1\}=\{0,1,2\}=_{\mathcal{I}} A$ for all $A \in \mathcal{I}$ then $\mathcal{I} \cap \Sigma(\mathcal{I})=\emptyset$. In particular, if $\mathcal{I}$ is a maximal ideal, that is, the complement of a free ultrafilter on $\mathbb{N}$, then $\mathcal{I}$ is not a meager subset of $\mathcal{P}(\mathbb{N})$, hence $\Sigma(\mathcal{I})$ is not comeager. 
However, the families $\mathcal{Z}_\alpha$ are $F_{\sigma\delta}$-subsets of $\mathcal{P}(\mathbb{N})$ for each $\alpha \in (0,1]$, hence they are meager, cf. \cite[Proposition 1.1]{MR3391516}. 

Second, S{\'{a}}rk{\"{o}}zy proved in \cite{MR146164} that there exists a constant $c>0$ such that, if $A\subseteq \mathbb{N}$ is infinite, then there exist a totally irreducible $B\in \Sigma_\star$ and $n_0 \in \mathbb{N}$ such that 
$$
\forall n\ge n_0, \quad |(A\bigtriangleup B)(n)| \le c \frac{|A(n)|}{\sqrt{\log \log  |A(n)|}}.
$$
In particular, since the function $t\mapsto t/\sqrt{\log\log t}$ 
is definitively increasing, it follows that $A=_{\mathcal{Z}}B$, therefore every equivalence class of $\mathcal{P}(\mathbb{N})/\mathcal{Z}$ contains a totally irreducible set.

Lastly, solving another conjecture of Erd{\H{o}}s, it has been shown in \cite{MR4042162, MR3919363} that, if $A\subseteq \mathbb{N}$ has positive upper asymptotic density, i.e., $A \notin \mathcal{Z}$, then there exist two infinite sets $B,C \subseteq \mathbb{N}$ such that $B+C\subseteq A$. Hence, even in the optimistic case that one could prove the comeagerness of $\Sigma(\mathcal{Z})$, cf. Section \ref{sec:conclusions}, there is no hope to show that also the smaller family of sets $A\subseteq \mathbb{N}$ such that, if $A^\prime=_{\mathcal{Z}}A$, then $A^\prime$ does not contain a nontrivial sumset is comeager. (Note that the same remark holds also in the measure sense: indeed, since almost all numbers are normal, then almost all subsets of $\mathbb{N}$ of them have asymptotic density $\frac{1}{2}$, hence almost all of them contains a sumset between two infinite sets.)

We conclude with an easy consequence of Theorem \ref{thm:meageranalogue}: 
\begin{cor}\label{cor:consequenceABmeager}
Fix two families $\mathcal{A}, \mathcal{B} \subseteq \mathcal{P}(\mathbb{N})$ containing $\{0\}$. Then 
$$
\mathcal{A}+\mathcal{B}:=\{A+B: A \in \mathcal{A}, B \in \mathcal{B}\}
$$
is meager if and only if both $\mathcal{A}$ and $\mathcal{B}$ are meager.
\end{cor}

\section{Proofs}

\begin{proof}
[Proof of Theorem \ref{thm:mainmeager}] 
Fix $\alpha \in (0,\frac{1}{3})$. 
We are going to use the Banach--Mazur game defined as follows, see \cite[Theorem 8.33]{MR1321597}\textup{:} 
Players I and II choose alternatively nonempty open subsets of $\{0,1\}^{\mathbb{N}}$ as a nonincreasing chain 
$$
U_0\supseteq V_0 \supseteq U_1 \supseteq V_1\supseteq \cdots, 
$$
where Player I chooses the sets $U_0,U_1,\ldots$; Player II is declared to be the winner of the game if 
\begin{equation}\label{eq:claimgame}
\bigcap\nolimits_{m\ge 0} V_m \cap \, \Sigma(\mathcal{Z}_{\alpha})\neq \emptyset. 
\end{equation}
Then Player II has a winning strategy (that is, he is always able to choose suitable sets $V_0, V_1,\ldots$ so that \eqref{eq:claimgame} holds at the end of the game) if and only if $\Sigma(\mathcal{Z}_{\alpha})$ is a comeager set in the Cantor space $\{0,1\}^{\mathbb{N}}$.  
Since $\mathcal{P}(\mathbb{N})$ is identified with $\{0,1\}^{\mathbb{N}}$, 
a basic open set in $\mathcal{P}(\mathbb{N})$ will be a cylinder of the type $\{A\subseteq \mathbb{N}: A(k)=F\}$ for some integer $k \in \mathbb{N}$ and some (possibly empty) finite set $F\subseteq [0,k]$.

At this point, we define define the strategy of player II recursively as it follows. Suppose that the nonempty open sets $U_0\supseteq V_0 \supseteq \cdots \supseteq U_{m}$ have been already chosen, for some $m \in \mathbb{N}$. Then there exists a finite set $F_m \in \mathrm{Fin}$ and an integer $k_m \ge \max (F_m\cup \{0\})$ such that 
$$
\forall A\subseteq \mathbb{N}, \quad 
A(k_m)=F_m \implies A \in U_m.
$$
Without loss of generality we can assume that $|F_m|\ge 2$. 
Using the continuity of the map $\beta \mapsto \alpha \beta$ and recalling that $\alpha<\frac{1}{3}$, we can fix a real $\beta \in (\frac{3}{4},1)$ such that $\alpha \beta <\frac{1}{4}$.

Thus, define $t_m:=\lfloor k_m^\beta\rfloor$ and 
$$
V_m:=\left\{A \subseteq \mathbb{N}: A(7k_m+t_m^2)=F_m \cup (k_m,2k_m] \cup \bigcup_{i=1}^{t_m}{\{5k_m+it_m\}}\right\}.
$$
In other words, a set $A\in U_m$ belongs to $V_m$ if it contains the block of integers $[k_m+1,2k_m]$ and, then, it is followed by an arithmetic progression of $t_m$ elements and distance $t_m$; note that each $A \in V_m$ ends with a further gap of lenght $2k_m$. 

Hence, by construction, $V_m$ is a nonempty open set contained in $U_m$. 
Finally, observe that there exists a unique set $A\subseteq \mathbb{N}$ such that 
$$
\{A\}=\bigcap\nolimits_{m\ge 0}V_m.
$$
Indeed, since the sequence $(k_m)_{m\ge 0}$ is strictly increasing, the definition of the sets $V_m$ gives us, in particular, all the finite truncations $A(k_m)$.

To complete the proof, we have to show that $A\in \Sigma(\mathcal{Z}_{\alpha})$. 
For, let us suppose for the sake of contradiction that there exist $A^\prime, B,C\subseteq \mathbb{N}$ such that $|B|,|C|\ge 2$ and
$$
A=_{\mathcal{Z}_{\alpha}}A^\prime=B+C.
$$
Let $m$ be a sufficiently large integer 
with the property that $|(A\bigtriangleup A^\prime)(n)|\le \frac{1}{2}n^\alpha$ for all $n\ge 2k_m$, which is possible since $A\bigtriangleup A^\prime \in \mathcal{Z}_{\alpha}$ (further properties will be specified in the course of the proof recalling simply that "$m$ is large"). 
Since $(k_m,2k_m]\subseteq A$ by construction, then 
\begin{displaymath}
\begin{split}
|A^\prime(2k_m)| 
&\ge \left|A\left(2k_m-\frac{1}{2}(2k_m)^\alpha\,\right)\right|\\
&\ge \left|(A \cap (k_m,2k_m-\frac{k_m^\alpha}{2^{1-\alpha}}\,]\right|
=\left\lfloor k_m\left(1-\frac{1}{(2k_m)^{1-\alpha}}\right)\right\rfloor 
\ge \frac{k_m}{2},
\end{split}
\end{displaymath}
where the last inequality holds since $m$ is large. 
On the other hand, $A^\prime(n)$ is contained in $B(n)+C(n)$ for all $n \in \mathbb{N}$, so that 
$$
|A^\prime(2k_m)| \le |B(2k_m)| \cdot |C(2k_m)|,
$$
which implies that 
$$
\max\left\{|B(2k_m)|, |C(2k_m)|\right\}\ge 
\sqrt{\frac{k_m}{2}}.
$$
Up to relabeling of the sets $B$ and $C$, we can assume without loss of generality that 
$|B(2k_m)|\ge |C(2k_m)|$. 

Now, observe that, since $m$ is large, 
\begin{equation}\label{eq:contrafinally}
|(A\bigtriangleup A^\prime)(7k_m+t_m^2)|\le \frac{1}{2}(7k_m+t_m^2)^\alpha 
\le t_m^{2\alpha}, 
\end{equation}
which is smaller than $\frac{1}{2}t_m$. 
This implies that there exists a subset $S_m\subseteq \{1,\ldots,t_m\}$ such that $|S_m|\ge \frac{1}{2}t_m$ and
$$
\{5k_m+it_m: i \in S_m\}\subseteq A^\prime.
$$
Hence, for each $i \in S_m$ there exist $b_i \in B$ and $c_i \in C$ such that $5k_m+it_m=b_i+c_i$. Therefore
$$
\forall i \in S_m, \quad 
\max\{b_i,c_i\} \ge \frac{5k_m+it_m}{2} > 2k_m.
$$

At this point, let us suppose that 
there exists $i \in S_m$ such that 
$c_i>2k_m$. 
It follows that 
$$
B(2k_m)+\{c_i\}\subseteq (B+C)\cap (2k_m, 7k_m+it_m]
\subseteq A^\prime \cap (2k_m,7k_m+t_m^2],
$$
so that, since $m$ is large, we have 
\begin{displaymath}
\begin{split}
|(A\bigtriangleup A^\prime)(7k_m+t_m^2)| &\ge |B(2k_m)|-\frac{2k_m}{t_m} \\
&\ge \sqrt{\frac{k_m}{2}}-\sqrt{\frac{k_m}{8}} = \sqrt{\frac{k_m}{8}}.
\end{split}
\end{displaymath}
However, this contradicts \eqref{eq:contrafinally}: indeed, since $2\alpha\beta<\frac{1}{2}$ and $m$ is large, we have 
\begin{equation}\label{eq:finalcontraaaaa}
|(A\bigtriangleup A^\prime)(7k_m+t_m^2)|\le t_m^{2\alpha} \le k_m^{2\alpha\beta} \le \sqrt{\frac{k_m}{16}}.
\end{equation}

This means that $c_i\le 2k_m<b_i$ for all $i \in S_m$. 
Let $i,j \in [1, t_m]$ such that $i-j \ge 4k_m^{1-\beta}$. 
Since $m$ is large, then 
\begin{displaymath}
\begin{split}
b_i-b_j=(i-j)t_m-c_i+c_j &\ge (i-j)t_m-2k_m \\
&\ge 4k_m^{1-\beta}t_m-2k_m \ge 3k_m-2k_m=k_m. 
\end{split}
\end{displaymath}
Hence there exist integers $1=i_1<i_2<\cdots<i_{q_m}\le t_m$ such that 
$$
\forall j=1,\ldots,q_m-1,\quad i_{j+1}-i_j \ge 4k_m^{1-\beta} 
\quad \text{ and }\quad 
b_{i_{j+1}}-b_{i_j} \ge k_m
$$
and, since $m$ is large, 
$$
q_m \ge \frac{t_m}{5k_m^{1-\beta}}\ge \frac{1}{6}k_m^{2\beta-1}.
$$

Let us call $c^\prime:=\min C$ and $c^{\prime\prime}:=\min C\setminus \{c^\prime\}$. Since $m$ is large, we can assume that $t_m \ge 2c^{\prime\prime}$. It follows that 
$$
b_{i_1}+c^\prime<b_{i_1}+c^{\prime\prime}<b_{i_2}+c^\prime<b_{i_2}+c^{\prime\prime}<\cdots<b_{i_{q_m}}+c^\prime<b_{i_{q_m}}+c^{\prime\prime}.
$$
Therefore we have $2q_m$ distinct elements in $(B+C) \cap (2k_m, 7k_m+t_m^2)=A^\prime \cap (2k_m, 7k_m+t_m^2)$ and 
at most half of them are nor equal to any $5k_m+ht_m$, $1\le h\le t_m$. Indeed since $c^{\prime\prime}-c^\prime \le \frac{1}{2}t_m$, $b_{i_j}+c^{\prime}$ and $b_{i_j}+c^{\prime\prime}$ cannot be together written as $5k_m+h^\prime t_m$ and $5k_m+h^{\prime\prime}t_m$, respectively. 
It follows that 
$$
|(A\bigtriangleup A^\prime)(7k_m+t_m^2)| \ge q_m \ge \frac{1}{6}k_m^{2\beta-1}
$$
which contradicts again \eqref{eq:finalcontraaaaa}, since $m$ is large and $2\beta-1>\frac{1}{2}$.
\end{proof}

\begin{proof}
[Proof of Theorem \ref{thm:mainmeasure}] 
Hereafter, denote explicitly by $h: \mathrm{Fin}^+\to (0,1]$ the bijection between the family 
$\mathrm{Fin}^+:=\mathcal{P}(\mathbb{N})\setminus \mathrm{Fin}$ of all infinite subsets of $\mathbb{N}$ and the set of reals in $(0,1]$ through their unique nonterminating dyadic expansions. 
Also, let $\Omega$ be the set of normal numbers in $(0,1]$.  
It follows by Borel's normal number theorem that $\Omega \in \mathscr{M}$ and $\lambda(\Omega)=1$, see e.g. \cite[Theorem 1.2]{MR1324786}. Observe that, if a set $A$ belongs to $\widehat{\Omega}:=h^{-1}[\Omega]$, then, by the definition of normal numbers, 
\begin{equation}\label{eq:kdsjhgkhf}
|\{j \in [0,n]: A(n+|F|)\cap (I+j)=F+j\}|=2^{-|I|}n+o(n)
\end{equation}
as $n\to \infty$, for all nonempty finite sets $F\subseteq I$ such that $I\subseteq \mathbb{N}$ is an interval containing $0$. 
Then the claim can be rewritten equivalently as
$$
\textstyle 
\lambda\left(h\left[\widehat{\Omega}\setminus \Sigma(\mathcal{Z})\right]\right)=0,
$$
and note that, by definition, 
$$
\widehat{\Omega}\setminus \Sigma(\mathcal{Z})=\left\{A\in \widehat{\Omega}: \exists A^\prime, B,C\subseteq \mathbb{N}, A=_{\mathcal{Z}}A^\prime=B+C \text{ and } |B|,|C|\ge 2\right\}.
$$

First, we claim that, if $A \in \widehat{\Omega}$ and $A=B+C$ for some $B,C \subseteq \mathbb{N}$ with $|B|,|C|\ge 2$, then both $B$ and $C$ need to be infinite sets. Indeed, suppose for the sake of contradiction that $B$ is a finite set and define $m:=1+\max B$. Since $A \in \widehat{\Omega}$ there exists an integer $a \in A$ bigger than $m$ such that $A \cap [a-m,a+m]=\{a\}$. However, since $a \in A$ there exist $b \in B$ and $c \in C$ such that $a=b+c$. Since $|B|\ge 2$ there exists $b^\prime \in B\setminus \{b\}$. This is a contradiction because $a^\prime:=b^\prime+c$ would be an integer in $A \cap [a-m,a+m]$ which is different from $a$. By symmetry, also $C$ needs to be infinite. 

Second, we claim that, if $A \in \widehat{\Omega}$ and $A=B+C$ for some infinite sets $B,C \subseteq \mathbb{N}$, then both $B$ and $C$ belong to $\mathcal{Z}$. For, let $(b_n: n \in \mathbb{N})$ be the increasing enumeration of the integers in $B$ and define $I_k:=[0, b_{k+1}-1]$ for all $k \in \mathbb{N}$. 
Fix $k \in \mathbb{N}$ and note that, if $c \in C(n)$ then $c+b_i \in (B+C)(n+b_{k+1})=A(n+b_{k+1})$ for all $i \in [0,k]$ and all $n \in \mathbb{N}$. 
Letting $\mathscr{S}_k$ be the family $\{S\subseteq \mathbb{N}: \{b_0,b_1,\ldots,b_k\}\subseteq S\subseteq I_k\}$, we obtain
$$
|C(n)|\le \sum_{S \in \mathscr{S}_k}\left|\{j \in [0,n+b_{k+1}]: A(n+b_{k+1}) \cap (I_k+j)=S+j\}\right|
$$
for all $n \in \mathbb{N}$. 
At this point, since $A \in \widehat{\Omega}$ and $|\mathscr{S}_k|=2^{|I_k|-(k+1)}$, it follows by \eqref{eq:kdsjhgkhf} that 
$$
|C(n)| \le |\mathscr{S}_k|\cdot \left(2^{-|I_k|}n+o(n)\right)\le 2^{-k}n
$$
for all sufficiently large $n \in \mathbb{N}$. 
By the arbitrariness of $k \in \mathbb{N}$, we conclude that $C \in \mathcal{Z}$ and, by symmetry, $B \in \mathcal{Z}$ as well. 

Third, note that, if $A \in \widehat{\Omega}$ and $A^\prime=_{\mathcal{Z}}A$, then, by the definition of normal numbers, $A^\prime \in \widehat{\Omega}$ as well. Putting everything together it follows the set $\widehat{\Omega}\setminus \Sigma(\mathcal{Z})$ can be rewritten equivalently as the family of all $A \in \widehat{\Omega}$ such that $A=_{\mathcal{Z}}A^\prime=B+C$ for some $A^\prime \in \widehat{\Omega}$ and some infinite $B,C \in \mathcal{Z}$. 
Therefore, by monotonicity, it is enough to check that $\lambda(h[\mathscr{A}])=0$, where 
$$
\mathscr{A}:=\left\{A\subseteq \mathbb{N}: \exists A^\prime\subseteq \mathbb{N}, \exists B,C\in \mathcal{Z}\cap \mathrm{Fin}^+,\,\, A=_{\mathcal{Z}}A^\prime=B+C\right\}.
$$
Let $k$ be a sufficiently large integer that will be chosen later (it will be enough to set $k=17$). Suppose that $A \in \mathscr{A}$ and pick $A^\prime \subseteq \mathbb{N}$ and infinite sets  $B,C \in \mathcal{Z}$ such that $A=_{\mathcal{Z}}A^\prime=B+C$. 
Then there exists $n_0=n_0(k)\in \mathbb{N}$ such that 
$A^\prime(n)=(B(n)+C(n)) \cap [0,n]$ and 
$$
\max\{|(A\bigtriangleup A^\prime)(n)|, |B(n)|, |C(n)|\}\le n/k
$$ 
for all $n\ge n_0$. At this point, for each $n \in \mathbb{N}$, let $\mathcal{E}_n$ be the family of all $X\subseteq [0,n]$ such that $\max\{|X\bigtriangleup X^\prime|, |Y|, |Z|\} \le n/k$ and $X^\prime=(Y+Z) \cap [0,n]$ for some $X^\prime, Y,Z\subseteq [0,n]$. 
Hence $A(n) \in \mathcal{E}_n$ for all $n\ge n_0$, which implies that 
\begin{equation}
\label{eq:inclusion}
\mathscr{A}
\subseteq
\bigcap_{n\ge n_0} \mathcal{E}_n
\subseteq
\bigcup_{m \in \mathbb{N}} \bigcap_{n\ge m} \mathcal{E}_n
\subseteq
\bigcap_{m \in \mathbb{N}} \bigcup_{n\ge m} \mathcal{E}_n.
\end{equation}
To conclude the proof, let us compute the probability $P(\mathcal{E}_n)$ of the event $\mathcal{E}_n$ with respect to the uniform probability measure $P$ on $[0,n]$. 
Observe that both $Y$ and $Z$ can be chosen in at most 
$$
w_{n,k}:=\sum_{i=0}^{n/k}\binom{n+1}{i}
$$
ways and, for each $X^\prime:=(Y+Z) \cap [0,n]$, the set $X$ can be obtained with at most $w_{n,k}$ modifications. Hence $X^\prime$ can be chosen in at most $w_{n,k}^2$ possibilities and, for each such $X^\prime$, its modification $X$ will be obtained in 
at most $w_{n,k}$ ways. 
Using Stirling's approximation, it follows that 
\begin{displaymath}
\begin{split}
P(\mathcal{E}_n) &\le \frac{1}{2^{n+1}}\cdot w_{n,k}^2 \cdot w_{n,k} 
\ll \frac{n^3}{2^n}\cdot \binom{n}{n/k}^3 \\
&\ll \frac{n^3}{2^n}\cdot \left(\frac{n^{n+\frac{1}{2}}}{(\frac{n}{k})^{\frac{n}{k}+\frac{1}{2}}\cdot (\frac{k-1}{k}\hspace{.4mm}n)^{\frac{k-1}{k}\hspace{.2mm}n+\frac{1}{2}}}\right)^3
=\frac{k^3}{(k-1)^{3/2}}\cdot \frac{n^{2}}{2^n}\cdot \alpha_k^{3n},
\end{split}
\end{displaymath}
as $n\to \infty$, where
$$
\alpha_k:=k^{1/k} \cdot \left(\frac{k}{k-1}\right)^{(k-1)/k}.
$$
Since $\alpha_t\to 1^+$ as $t\to \infty$, we can fix an integer $k \in \mathbb{N}$ for which $\alpha_k \in (1,2^{1/3})$. 
Hence there exists $c \in (1/2,1)$ for which 
$$
P(\mathcal{E}_n) \le c^n
$$ 
for all sufficiently large $n$. Since $\sum_n P(\mathcal{E}_n)<\infty$, it follows by Borel--Cantelli lemma and inclusion \eqref{eq:inclusion} that $\lambda(h[\mathscr{A}])=0$, which concludes the proof. 
\end{proof}

\begin{proof}
[Proof of Corollary \ref{cor:consequenceABmeager}]
First, suppose that $\mathcal{A}$ is not meager (the case $\mathcal{B}$ not meager is analogous). Then $\mathcal{A}+\mathcal{B}$ contains $\mathcal{A}+\{\{0\}\}=\mathcal{A}$, so that $\mathcal{A}+\mathcal{B}$ is not meager. 

Conversely, suppose that both $\mathcal{A}$ and $\mathcal{B}$ are meager, and note that
\begin{displaymath}
\begin{split}
\mathcal{A}+\mathcal{B}\subseteq 
& \bigcup\nolimits_{a \in \mathbb{N}: \{a\} \in \mathcal{A}}(\{a\}+\mathcal{B})\cup 
\bigcup\nolimits_{b \in \mathbb{N}: \{b\} \in \mathcal{B}}(\mathcal{A}+\{b\})\cup 
(\mathcal{P}(\mathbb{N})\setminus \Sigma).
\end{split}
\end{displaymath}
The claim follows by the fact that all sets $\{a\}+\mathcal{B}$ and  $\mathcal{A}+\{b\}$ are meager, and that $\mathcal{P}(\mathbb{N})\setminus \Sigma$ is meager as well by Theorem \ref{thm:meageranalogue}.
\end{proof}

\section{Concluding Remarks and Open Questions}\label{sec:conclusions}

In the same spirit of \cite[Section 6]{Bien2022b}, the statement of Theorem \ref{thm:mainmeager} holds also replacing $\mathbb{N}$ with a numerical submonoid of $\mathbb{N}$, that is, a pair $(M,+)$ where $M$ is a cofinite subset of $\mathbb{N}$ closed under sum. Indeed, the very same proof holds substituting the definition of $k_0=\max F_0$ with $k_0=\max (F_0 \cup M^c)$.

We leave as an open question for the interested reader to check whether Theorem \ref{thm:mainmeager} can be strenghtened to prove the comeagerness of 
$\Sigma(\mathcal{Z}_{1/2})$, on the same lines of Erd{\H{o}}s' conjecture, or even of the smaller subset $\Sigma(\mathcal{Z})$, in analogy with Theorem~\ref{thm:mainmeasure}.

Lastly, we conclude with an evocative question: is it true that every (set identified with a) normal number is not a nontrivial sumset? With the notation of the proof of Theorem \ref{thm:mainmeasure}, this amounts to ask whether the inclusion $\widehat{\Omega}\subseteq \Sigma$ holds.

\subsection{Acknowledgements} 
The author is grateful to an anonymous referee for a careful reading of the manuscript and several useful suggestions.

\bibliographystyle{amsplain}

\end{document}